\documentclass[11pt,reqno]{amsart}
\usepackage {amssymb}
\usepackage {amsmath}
\usepackage {bbm}
\usepackage{amsthm}
\usepackage{mathtools}
\usepackage{graphicx}
\usepackage {amscd}
\usepackage {epic}
\usepackage{array}
\usepackage{booktabs}
\usepackage{tabularx}
\usepackage{mathrsfs}

\usepackage{etaremune}

\DeclareFontFamily{U}{dutchcal}{\hyphenchar\font=-1}
\DeclareFontShape{U}{dutchcal}{m}{n}{ <-> dutchcal-r}{}
\DeclareSymbolFont{dutchletters}{U}{dutchcal}{m}{n}

\DeclareMathSymbol{\mathdutchcal}{0}{dutchletters}{"41} 
\newcommand{\dcal}[1]{\text{\usefont{U}{dutchcal}{m}{n}#1}}

\usepackage[dvipsnames]{xcolor}
\usepackage[colorlinks,citecolor=OliveGreen,linkcolor=Mahogany,urlcolor=Plum,pagebackref]{hyperref}
\usepackage[alphabetic]{amsrefs}
\usepackage{cleveref}

\usepackage{enumerate}
\usepackage{tikz}
\usepackage{tikz-cd}
\usetikzlibrary{arrows.meta}
\usepackage{verbatim,color,geometry}
\usepackage[all]{xy}
\usepackage{enumitem}
\usepackage{bm}
\usepackage{longtable}
\usepackage{aliascnt}
\usetikzlibrary{calc}
\usepackage{textcomp}
\usepackage{tablefootnote}
\usepackage{float}

\newcommand{\ove}[1]{\overline{#1}}
\newcommand{\wt}[1]{\widetilde{#1}}

\newcommand{\cC}{\dcal{C}}

\newcommand{\cE}{\dcal{E}}

\newcommand{\cG}{\dcal{G}}
\newcommand{\cH}{\dcal{H}}
\newcommand{\cI}{\dcal{I}}

\newcommand{\cK}{\dcal{K}}
\newcommand{\cL}{\dcal{L}}
\newcommand{\cM}{\dcal{M}}
\newcommand{\cN}{\dcal{N}}
\newcommand{\cO}{\dcal{O}}

\newcommand{\cQ}{\dcal{Q}}

\newcommand{\cS}{\dcal{S}}

\newcommand\MM{{\mathcal{M}}}

\newcommand{\bP}{\mathbb{P}}
\newcommand{\bC}{\mathbb{C}}

\newcommand{\bZ}{\mathbb{Z}}

\newcommand{\bG}{\mathbb{G}}

\newcommand{\bH}{\mathbb{H}}

\newcommand{\Lie}{\mathrm{Lie}}

\newcommand{\coker}{\mathrm{Coker}}

\newcommand{\Gr}{\mathrm{Gr}}

\newcommand{\sX}{\mathscr{X}}

\DeclareMathOperator{\Aut}{Aut}
\DeclareMathOperator{\Stab}{Stab}

\DeclareMathOperator{\divi}{div}

\DeclareMathOperator{\Exc}{Exc}

\DeclareMathOperator{\Hilb}{Hilb}

\DeclareMathOperator{\GL}{GL}

\DeclareMathOperator{\im}{Im}

\DeclareMathOperator{\sing}{sing}
\DeclareMathOperator{\univ}{univ}
\DeclareMathOperator{\Supp}{Supp}

\DeclareMathAlphabet{\mathbbb}{U}{bbold}{m}{n}

\newcommand{\tC}{\widetilde{C}}

\newcommand{\sC}{\mathscr{C}}

\newcommand{\rk}{\mathrm{rk}}

\newcommand{\tensor}{\otimes}

\numberwithin{equation}{section}

\newtheorem{prop}{Proposition}[section]

\newcommand{\newaliastheorem}[3]{%
  \newaliascnt{#1}{prop}%
  \newtheorem{#1}[#1]{#2}%
  \aliascntresetthe{#1}%
  \crefname{#1}{#2}{#3}%
  \Crefname{#1}{#2}{#3}%
}

\newaliastheorem{thm}{Theorem}{Theorems}
\newaliastheorem{theorem}{Theorem}{Theorems}
\newaliastheorem{lem}{Lemma}{Lemmas}
\newaliastheorem{lemma}{Lemma}{Lemmas}
\newaliastheorem{cor}{Corollary}{Corollaries}
\newaliastheorem{corollary}{Corollary}{Corollaries}
\newaliastheorem{theodef}{Theorem-Definition}{Theorem-Definitions}
\newaliastheorem{prop-def}{Proposition-Definition}{Proposition-Definitions}
\newaliastheorem{convention}{Convention}{Conventions}
\newaliastheorem{conj}{Conjecture}{Conjectures}
\newaliastheorem{conjecture}{Conjecture}{Conjectures}

\theoremstyle{definition}
\newaliastheorem{defi}{Definition}{Definitions}
\newaliastheorem{defn}{Definition}{Definitions}
\newaliastheorem{exa}{Example}{Examples}
\newaliastheorem{exam}{Example}{Examples}
\newaliastheorem{expl}{Example}{Examples}
\newaliastheorem{rem}{Remark}{Remarks}
\newaliastheorem{rmk}{Remark}{Remarks}
\newaliastheorem{remark}{Remark}{Remarks}
\newaliastheorem{que}{Question}{Questions}
\newaliastheorem{pro}{Problem}{Problem}

\crefname{prop}{Proposition}{Propositions}
\Crefname{prop}{Proposition}{Propositions}

\title{Algebraic hyperbolicity of very general hypersurfaces in projective spaces}

\author{Sixuan Lou}
\address{Department of Mathematics Statistics and Computer Science, University of Illinois Chicago, 851 S Morgan St, Chicago, IL 60607, USA}
\email{tmzl.sx@gmail.com}

\author{Junyan Zhao}
\address{Department of Mathematics, University of Maryland, 4176 Campus Dr, College Park, MD 20742, USA}
\email{jzhao81@umd.edu}

\date{}

\begin{document}

\begin{abstract}
We prove that a very general sextic threefold in $\bP^4$ is algebraically hyperbolic, settling the last open case and completing the classification of algebraic hyperbolicity for very general hypersurfaces in projective space. The proof follows the Coskun--Riedl scroll construction, with the moduli of stable maps and the geometric canonical height as new ingredients.
\end{abstract}
\maketitle

\tableofcontents

\section{Introduction}

A complex projective variety $X\subseteq\bP^n$ is \emph{algebraically hyperbolic} if there exists a constant $\epsilon>0$ such that
\[
2p_g(C)-2\ \geq\  \epsilon\deg C
\]
for every integral curve $C\subseteq X$. In particular, algebraically hyperbolic varieties contain no rational or elliptic curves.

Algebraic hyperbolicity is closely related to the complex analytic notion of Brody hyperbolicity introduced in \cite{Br78}, which characterizes the nondegeneracy of the Kobayashi metric. A complex variety $X$ is \emph{Brody hyperbolic} if every holomorphic map
$\bC\to X$ is constant. It is proved in \cite{Dem97} that Brody hyperbolicity implies algebraic hyperbolicity, and conjecturally the two notions are equivalent for smooth projective varieties.

\subsection{Algebraic hyperbolicity for hypersurfaces}
Very general hypersurfaces form a natural testing ground for deep conjectures on hyperbolicity. This leads to the following fundamental problem.

\begin{pro}
Determine the algebraic hyperbolicity of a very general hypersurface $X_d\subseteq\bP^n$ of degree $d$.
\end{pro}

This problem has been studied extensively; see \Cref{sec:History} for a
detailed account. The only case left unresolved by the previous results is that of a very general sextic threefold in $\bP^4$, which is settled
by the main theorem of this paper.

\begin{theorem}\label{thm:main-intro}
A very general sextic threefold $X_6\subseteq\bP^4$ is algebraically
hyperbolic.
\end{theorem}

Together with the previous results, \Cref{thm:main-intro} completes the
classification of algebraic hyperbolicity for very general hypersurfaces
in projective space.

\smallskip

It would be interesting to investigate whether the moving linear subspaces technique can be adapted to some of the remaining low dimensional cases for very general complete intersections; see
\cite[Corollary~B]{BP26}.

\medskip

\subsection{Sketch of the proof}
Let $X\subseteq\bP^4$ be a very general sextic threefold, and let
$f\colon\tC\to X$ be the normalization map of an integral curve. A standard equivariant deformation framework for curves on very general hypersurfaces produces a family of linear subspaces of $\bP^4$ parametrized by $\tC$. The resulting scroll is swept out either by lines or by planes. The line scroll case can be treated using techniques already available in the literature; see e.g. \cite[Remark~3.5]{Yeong25}. For the reader's convenience, we provide a
self-contained proof in \Cref{sec:line-scroll}, presenting the argument from a more geometric perspective.

\smallskip
The main case is that of a scroll of planes. The construction associates with $f$ a morphism
\[
\phi\colon\tC\longrightarrow\bG(2,\bP^4),
\qquad
p\longmapsto\Lambda_p,
\]
where $\Lambda_p\subseteq\bP^4$ is a plane containing $f(p)$. The cases in which $\phi$ is constant or $f(p)$ is a singular point of
$X\cap\Lambda_p$ for every $p\in\tC$ are treated separately. We may
therefore restrict our attention to the generic case.

\smallskip
After a finite base change $C'\to\tC$, the normalization of the
pullback of $\cC_{\bG}$ extends, by the properness of the moduli stack
of Kontsevich stable maps, to a family of prestable curves equipped with a morphism to $X$. Its fibers are in fact DM stable. The
points $f(p)\in X\cap\Lambda_p$ define a section $\Sigma$, yielding a surface $S$ of general type fibered in stable curves over $C'$. Using the section $\Sigma\simeq C'$, one can bound
$\deg f^*\cO_{\bP^4}(1)$ in terms of numerical invariants of the
fibration, including the geometric canonical height $(K_{S/C'}\cdot\Sigma)$. The problem is thereby reduced to the geometry of a fibered surface of general type. The logarithmic Bogomolov--Miyaoka--Yau inequality bounds $(K_{S/C'}\cdot\Sigma)$ in terms of the genera of the fibers and of $C'$, together with the number of singular fibers.

\smallskip

To control the number of singular fibers and the ramification of
$C'\to\tC$, we consider the universal family of plane sections
\[
\cC_{\bG}\longrightarrow\bG(2,\bP^4),
\qquad
\cC_{[\Lambda]}=X\cap\Lambda.
\]
After stratifying $\bG$, this family admits a simultaneous
normalization over each stratum. For every stratum $\bG_i$, we carry
out the preceding constructions universally, obtaining bounds for both the ramification of $C'\to\tC$ and the number of singular fibers in terms of $2g(\tC)-2$, whose coefficients depend only on the fixed data associated with $\bG_i$,
whenever $\phi(\eta_{\tC})\in\bG_i$. Since there are only finitely many
strata, these bounds yield the desired uniform genus--degree inequality and hence the algebraic hyperbolicity of $X$.

\smallskip

\subsection{History and prior work}\label{sec:History}

The study of curves on very general hypersurfaces has a long history, and the principal results are summarized in the following table.{
\small
\setlength{\tabcolsep}{6pt}
\renewcommand{\arraystretch}{1.25}

\begin{longtable}{
  @{}
  p{0.24\textwidth}
  p{0.72\textwidth}
  @{}
}
\toprule
Reference & Result \\
\midrule
\endfirsthead

\toprule
Reference & Result \\
\midrule
\endhead

\endfoot

\bottomrule

\caption{History of Algebraic hyperbolicity of hypersurfaces in $\bP^n$.}
\label{tab:hyperbolicity-history}\\
\endlastfoot

\cite[Theorem~8(i)]{BVdV78}
&
every $X_d\subseteq \bP^n$ with $d\leq 2n-3$ contains lines.
\\
\midrule

Bogomolov--Mumford \cite[Appendix]{MM83}
&
every quartic K3 surface $S\subseteq \bP^3$ contains a rational curve.
\\
\midrule

\cite{Xu94}
&
very general surfaces
$X_d\subseteq \bP^3$ with $d\geq 6$ are hyperbolic.
\\
\midrule

\cite{Voi96}
&
very general hypersurfaces
$X_d\subseteq \bP^n$ with $n\geq 4$ and $d\geq 2n-1$ are hyperbolic.
\\
\midrule

\cite{CR04,Pac04}
&
very general hypersurfaces
$X_d\subseteq \bP^n$ with $n\geq 6$ and $d= 2n-2$ are hyperbolic.
\\
\midrule

\cite[Corollary 1.5]{CR19}
&
very general quintic surfaces
$X_5\subseteq \bP^3$ are hyperbolic.
\\
\midrule

\cite[Theorem 3.15]{Yeong25}
&
very general octic fourfold
$X_8\subseteq \bP^5$ are hyperbolic.
\\
\midrule

\Cref{thm:main-intro}
&
very general sextic threefold
$X_6\subseteq \bP^4$ are hyperbolic.\\

\end{longtable}
}

The increasing difficulty of these cases is closely related to their
proximity to the Calabi--Yau threshold. By adjunction, one has
\[
K_{X_d}\ =\ \cO_{X_d}(d-n-1).
\]
Along the critical range
$d=2n-2$, the distance from this threshold is $n-3$. Thus, as the dimension decreases, the positivity available from the canonical bundle becomes progressively weaker. A sextic threefold satisfies $K_X=\cO_X(1)$, making it the most delicate case in the classification.

\subsection*{Conventions and notation}

We adopt the following conventions throughout the paper.

\begin{itemize}
    \item We work over the field $\bC$ of complex numbers.
    \item Following Grothendieck's convention, all projective spaces,
    projective bundles, and Grassmannians parameterize quotients.
    \item For a fixed $n$, we set
$V_d\coloneqq H^0(\bP^n,\cO_{\bP^n}(d))$; this notation will be used primarily for $n=4$.
\end{itemize}

\section{Preliminaries}

\subsection{Deformation framework}\label{sec:deformation}

We recall the deformation framework originating in \cite{Cle86,Ein88}, following the formulation used in
\cite[Section~2]{CR23}.

\smallskip

Suppose that a very general hypersurface of degree $d$ in $\bP^n$
contains an integral curve of geometric genus $g$ and degree $e$. Set $V_d\coloneqq H^0\bigl(\bP^n,\cO_{\bP^n}(d)\bigr)$, and let
\[
\sX^{\univ}
\ \coloneqq\ 
\left\{
(F,x)\in V_d\times\bP^n
\ \bigm|\
F(x)=0
\right\}
\]
be the universal family of hypersurfaces of degree $d$. Denote the two
projections by
\[
p_1\colon\sX^{\univ}\longrightarrow V_d,
\qquad
p_2\colon\sX^{\univ}\longrightarrow\bP^n.
\] Let $\Hilb(\bP^n)$ be the Hilbert scheme of curves in $\bP^n$, and $T\subseteq V_d\times\Hilb(\bP^n)$ be an irreducible locally closed
subvariety satisfying the following properties:
\begin{enumerate}
    \item a general point $(F,[C])\in T$ parametrizes an integral curve
    $C\subseteq X_F$ of geometric genus $g$ and degree $e$;
    \item the projection $T\to V_d$ is dominant;
    \item $T$ is invariant under the natural action of
    $G\coloneqq\GL_{n+1}$.
\end{enumerate} Let
\[
\pi_1\ \colon \ \sX\ \coloneqq\ T\times_{V_d}\sX^{\univ}\ \longrightarrow \ T,
\]
and let $\sC\subseteq\sX$ be the corresponding universal curve. By a standard argument, up to replacing $T$ by a $G$-invariant locally closed subvariety, we may assume that $T\to V_d$ is \'{e}tale and that there exists a smooth projective
family
\[
\varpi\colon\wt{\sC}\ \longrightarrow\  T
\]
of curves of genus $g$, together with a morphism
\[
h\colon\wt{\sC}\ \longrightarrow \ \sC\ \subseteq \ \sX
\]
over $T$, which fiberwise is a normalization of its image. Denote the natural projection by 
\[
\pi_2\colon\sX\longrightarrow\bP^n.
\]
Since $\wt{\sC}$ is $G$-equivariant and $G$ acts transitively on $\bP^n$, the
differential of $\pi_2\circ h$ is surjective. Let $T_{\sX/\bP^n}$ and $T_{\wt{\sC}/\bP^n}$ be the relative tangent bundles. 

For $t\in T$, let $\tC_t$ and $X_t$ be the fibers of $\wt{\sC}$ and $\sX$ over
$t$, respectively, and write
\[
h_t\colon \tC_t\longrightarrow X_t
\]
for the restriction of $h$. We define the normal sheaves
$N_{h/\sX}$ and $N_{h_t/X_t}$ by
\[
\begin{tikzcd}[column sep=small,row sep=small]
0 \arrow[r]
& T_{\wt{\sC}} \arrow[r]
& h^*T_{\sX} \arrow[r]
& N_{h/\sX} \arrow[r]
& 0, \\
0 \arrow[r]
& T_{\tC_t} \arrow[r]
& h_t^*T_{X_t} \arrow[r]
& N_{h_t/X_t} \arrow[r]
& 0.
\end{tikzcd}
\]

Let $\cM_d$ be the Lazarsfeld--Mukai bundle on $\bP^n$ associated with $\cO_{\bP^n}(d)$, defined by
\[
0
\ \longrightarrow  \ \cM_d
\ \longrightarrow \ V_d\otimes\cO_{\bP^n}
\ \longrightarrow\ \cO_{\bP^n}(d)
\ \longrightarrow\ 0.
\]
Its fiber at $q\in\bP^n$ is naturally identified with
$H^0\big(\bP^n,\cI_q\otimes\cO_{\bP^n}(d)\big)$.

\begin{prop}[{\textup{cf. \cite[Proposition~2.1]{CR23}}}]
\label{prop:normal-sheaf-identity}
There are natural isomorphisms
\begin{enumerate}
    \item[\textup{(1)}]
    $N_{h/\sX}|_{C_t}\simeq N_{h_t/X_t}$;
    \item[\textup{(2)}] $T_{\sX/\bP^n}\simeq\pi_2^*\cM_d$;
    \item[\textup{(3)}]
    $N_{h/\sX}\simeq \coker
    \big(T_{\wt{\sC}/\bP^n}
    \to
    h^*T_{\sX/\bP^n}\big)$.
\end{enumerate}
\end{prop}

The multiplication maps give a commutative diagram
\[
\begin{tikzcd}[ampersand replacement=\&]
0 \arrow[r]
  \& \cM_1\otimes V_{d-1} \arrow[r] \arrow[d]
  \& V_1\otimes V_{d-1}\otimes\cO_{\bP^n} \arrow[r] \arrow[d]
  \& \cO_{\bP^n}(1)\otimes V_{d-1} \arrow[d]\arrow[r] \&0 \\
0 \arrow[r]
  \& \cM_d \arrow[r]
  \& V_d\otimes\cO_{\bP^n} \arrow[r]
  \& \cO_{\bP^n}(d)\arrow[r] \& 0.
\end{tikzcd}
\]

\begin{lemma}\label{lem:surjection}
For every $d\geq 1$, the natural morphism
\[
\cM_1\otimes V_{d-1}\
\longrightarrow\  \cM_d
\]
is surjective.
\end{lemma}

\begin{proof}
At a point $q\in\bP^n$, the induced map on fibers is the multiplication
map
\[
H^0\bigl(\bP^n,\cI_q(1)\bigr)\otimes
H^0\bigl(\bP^n,\cO_{\bP^n}(d-1)\bigr)
\longrightarrow
H^0\bigl(\bP^n,\cI_q(d)\bigr).
\]
This map is surjective because the homogeneous ideal $\cI_q$ is generated
in degree one. 
\end{proof}

\subsection{Geometry of very general sextic threefolds}

We establish two basic properties of a very general sextic threefold that will be used repeatedly throughout the paper.

\begin{lemma}\label{lem:no-low-degree-plane-curve}
Every plane section of a very general sextic threefold $X\subseteq\bP^4$
is integral.
\end{lemma}

\begin{proof}
It suffices to show that a very general sextic
threefold contains no plane curve of degree at most $3$. The families of lines, plane conics, and plane cubics in $\bP^4$ have dimensions $6$, $11$, and $15$, respectively, whereas containing a fixed
such curve imposes $7$, $13$, and $18$ independent linear conditions on a sextic hypersurface. The standard incidence correspondence therefore shows that a very general sextic threefold contains none of these curves.
\end{proof}

\begin{prop}\label{prop:normalization-of-plane-curve}
For a very general sextic threefold $X\subseteq\bP^4$, every plane section
$X\cap\Lambda$ has geometric genus at least $4$.
\end{prop}

\begin{proof}
By \Cref{lem:no-low-degree-plane-curve}, every plane section is integral.
A plane sextic has arithmetic genus $10$. The locus of integral plane
sextics of geometric genus at most $3$ has codimension at least
$10-3=7$ in the space of plane sextics. Since the family of planes in
$\bP^4$ has dimension $6$, a standard incidence correspondence argument and dimension count show that a very
general sextic threefold has no such plane section. Hence every plane section has geometric genus at least $4$.
\end{proof}

\subsection{Fibered surfaces of log general type}

We record a result on the geometry of fibered surfaces of general type that will be used in \Cref{sec:moving-plane-case}. 

\begin{lemma}[\textup{cf. \cite{Tan95}}]\label{lem:log-BMY}
Let $\pi\colon S\longrightarrow C$ be a relatively minimal semistable fibration from a smooth projective surface to a smooth projective curve of genus $g\geq 2$, whose
fibers have arithmetic genus $h\geq 2$. Let $\Sigma\subseteq S$ be a section of $\pi$, and let $F_1,\ldots,F_s$ be the singular
fibers. Then
\[
(K_{S/C}\cdot\Sigma)
\ \leq\
(2h-1)\bigl(2g-2+s\bigr)-K_{S/C}^2.
\] In particular, one has $(K_{S/C}\cdot\Sigma)
\ \leq\
(2h-1)\bigl(2g-2+s\bigr)$.
\end{lemma}
\begin{proof}
Since $g\geq 2$, every rational curve on $S$ is contracted by $\pi$. Thus, by relative minimality and the cone theorem, $K_S=K_{S/C}+\pi^*K_C$ is nef. By the Arakelov--Beauville nefness theorem (cf. \cites{Ara71,Deb82}) $K_{S/C}$ is also nef. Therefore one has $(K_S^2)>0$ and $S$ is a surface of general type. We set
\[
p_i\ \coloneqq\ \pi(F_i),\qquad
D\ \coloneqq\ \Sigma+\sum_{i=1}^s F_i,
\qquad
d\ \coloneqq\ \deg\Bigl(K_C+\sum_{i=1}^s p_i\Bigr)
\ =\ 2g-2+s.
\]
Then $D$ is a semistable curve and $(S,D)$ is a log-smooth pair. By adjunction, one has
\[
(K_{S/C}\cdot\Sigma)+\Sigma^2
\ =\ \deg K_{\Sigma}-\deg K_C
\ =\ 0.
\]
By definition, one has
\[
K_S+D
\ =\
K_{S/C}+\Sigma+
\pi^*\Bigl(K_C+\sum_{i=1}^s p_i\Bigr).
\]
It follows that
\[
\begin{aligned}
(K_S+D)^2
&\ =\
\Bigl(
K_{S/C}+\Sigma+
\pi^*\Bigl(K_C+\sum_{i=1}^s p_i\Bigr)
\Bigr)^2\\
&\ =\
K_{S/C}^2+\Sigma^2+2(K_{S/C}\cdot\Sigma)
+2(2h-2)d+2d\\
&\ =\
K_{S/C}^2+(K_{S/C}\cdot\Sigma)+2(2h-1)d.
\end{aligned}
\]
By the logarithmic Bogomolov--Miyaoka--Yau inequality (cf. \cite[Theorem~7.6]{Sak80}), one has
\[
(K_S+D)^2
\ \leq\
3c_2\bigl(\Omega_S^1(\log D)\bigr).
\]
By the logarithmic Gauss--Bonnet formula,
\[
c_2\bigl(\Omega_S^1(\log D)\bigr)
\ =\
\chi_{\mathrm{top}}(S\setminus D)
\ =\ (2-2g-s)(1-2h)
\ =\ d(2h-1),
\]
since $S\setminus D\to C\setminus\{p_1,\ldots,p_s\}$ is a topologically locally trivial smooth fibration. Combining these equalities and inequalities, one obtains
\[
K_{S/C}^2+(K_{S/C}\cdot\Sigma)+2(2h-1)d
\ \leq\
3(2h-1)d,
\]
and therefore the desired inequality. The last statement follows immediately from the nefness of $K_{S/C}$.
\end{proof}

\bigskip

\section{Deformation and Scroll construction}

In this section, we specialize the construction of
\Cref{sec:deformation} to sextic threefolds. For a very general sextic threefold $X$ and every curve on $X$, we
associate a family of linear subspaces of $\bP^4$ of dimension either $1$ or $2$. We prove the desired hyperbolicity in the
line scroll case and reduce the plane scroll case to its generic case, which we will treat in \Cref{sec:moving-plane-case}.

\medskip

Retain the notation as in \Cref{sec:deformation} with $n=4$ and $d=6$. Let $t\in T$ be a very general point, $X$ (resp. $C$ and $\tC$) be the fiber $\sX_t$ (resp. $\sC_t$ and $\wt{\sC}_t$). Denote by $f:\tC\to X$ the morphism $h_t:\wt{\sC}_t\to \sX_t$, and set 
\[\cM\coloneqq f^*\cM_1,\qquad
\cL\coloneqq f^*\cO_{\bP^4}(1),\qquad
V \coloneqq H^0(\bP^4, \cO_{\bP^4}(1))\]

By \Cref{prop:normal-sheaf-identity,lem:surjection} one has a natural surjective morphism
\[
\alpha : \cM\otimes_{\bC} V_5 \ \longrightarrow \  N_{f / X}.
\]
As $N_{h/X}$ is of rank $2$, for a general $F_5\in V_5$, the image of $\alpha_{F_5}:\cM\to N_{h/X}$, denoted by $\cN$, is of rank either $1$ or $2$. Let $\cK$ be the kernel of $\cM\to \cN$. Then one has the following commutative diagram of exact sequences
\begin{equation}\label{eq:quotient-bundle}
    \begin{tikzcd}[ampersand replacement=\&]
	\& \cK \& \cK \&\& \\
	0 \& \cM \& {V\otimes\cO_{\tC}} \& \cL \& 0 \\
	0 \& \cN \& \cE \& \cL \& 0
	\arrow["{=}"{description}, draw=none, from=1-2, to=1-3]
	\arrow[from=1-2, to=2-2]
	\arrow[from=1-3, to=2-3]
	\arrow[from=2-1, to=2-2]
	\arrow[from=2-2, to=2-3]
	\arrow[two heads, from=2-2, to=3-2]
	\arrow[from=2-3, to=2-4]
	\arrow[two heads, from=2-3, to=3-3]
	\arrow[from=2-4, to=2-5]
	\arrow["{=}"{marking, allow upside down}, draw=none, from=2-4, to=3-4]
	\arrow[from=3-1, to=3-2]
	\arrow[from=3-2, to=3-3]
	\arrow[from=3-3, to=3-4]
	\arrow[from=3-4, to=3-5]
\end{tikzcd}
\end{equation} for some locally free sheaf $\cE$ on $\tC$. In particular, $\cE$ gives rise to a morphism \[\phi:\tC \ \longrightarrow \ \bG\coloneqq \Gr(V,\rk \cE). \]

\subsection{Rank one quotient and surface scroll}\label{sec:line-scroll}

In this subsection, we treat the case $\rk \cE=2$, i.e. when $\cN=\im\big(\alpha_{F_5}:\cM\to N_{f/X}\big)$ is of rank $1$ for a general $F_5$. This case was stated in
\cite[Remark~3.5]{Yeong25}, using methods similar to those adopted there. For the reader's convenience, we provide a proof in the specific case of sextic threefolds.

\smallskip

Let $p \in \tC$ be a general point, $q \coloneqq f(p) \in X$, and $\ell_p \subseteq \bP^4$ be the line represented by $\phi(p)$. Then $\ell_p$ is cut out by linear forms in $\cK_p$. By \cite[Lemma 2.2(i)]{Cle03}, the kernel $\cK_p$ is independent of the choice of quintic polynomial $F_5$.

\begin{lemma}
\label{lem:ideal-in-relative-tangent}
The space
$H^0\bigl(\bP^4,\cI_{\ell_p}(6)\bigr)$
is naturally contained in $T_{\wt{\sC}/\bP^4}|_{(t,p)}$.
\end{lemma}

\begin{proof}
By \Cref{prop:normal-sheaf-identity}, one has
\[
(N_{f/X})_p
\ =\
(N_{h_t/X_t})_p
\ \simeq\
(\cM_6)_q\big/T_{\wt{\sC}/\bP^4}|_{(t,p)}.
\] For a general $F_5\in V_5$, multiplication by $F_5$ induces a map $(\cM_1)_q\to(\cM_6)_q$. Since
$\cK_p$ is the kernel of the induced map
$\cM_p\to(N_{f/X})_p$, the image of the multiplication map
\[
\cK_p\otimes V_5
\ \simeq\ 
H^0\bigl(\bP^4,\cI_{\ell_p}(1)\bigr)
\otimes H^0\bigl(\bP^4,\cO_{\bP^4}(5)\bigr)
\ \longrightarrow\ 
(\cM_6)_q
\ =\ 
H^0\bigl(\bP^4,\cI_q(6)\bigr)
\]
is contained in $T_{\sC/\bP^4}|_{(t,p)}$. This image is precisely
$H^0\bigl(\bP^4,\cI_{\ell_p}(6)\bigr)$.
\end{proof}

To summarize, there is a commutative diagram of exact sequences: \begin{equation}\label{eq:comm-diag}
    \begin{tikzcd}[ampersand replacement=\&]
	\&\&\& {H^0\big(\ell,\cI_{q}(6)\big)} \\
	{T_{\wt{\sC}/\bP^4}|_{(t,p)}} \&\& {(\cM_6)_q\simeq H^0\big(\bP^4,\cI_q(6)\big)} \& {N_{f/X}|_p} \\
	\& {H^0\big(\bP^4,\cI_{\ell}(6)\big)} \\
	\& {\cK_p\simeq H^0\big(\bP^4,\cI_{\ell}(1)\big)} \& {\cM_p\simeq H^0\big(\bP^4,\cI_{q}(1)\big)} \& {\cN_p}
	\arrow["\beta"{description}, from=2-1, to=1-4]
	\arrow[hook, from=2-1, to=2-3]
	\arrow[two heads, from=2-3, to=1-4]
	\arrow[two heads, from=2-3, to=2-4]
	\arrow[dashed, hook', from=3-2, to=2-1]
	\arrow[hook, from=3-2, to=2-3]
	\arrow[dotted, from=4-2, to=2-1]
	\arrow["{\cdot F_5}", from=4-2, to=3-2]
	\arrow[hook, from=4-2, to=4-3]
	\arrow["{\cdot F_5}", from=4-3, to=2-3]
	\arrow[two heads, from=4-3, to=4-4]
	\arrow[from=4-4, to=2-4]
\end{tikzcd}.\end{equation} Here the middle and bottom rows and the diagonal sequence are exact.
\begin{comment}
    We will prove that $\ell_p$ is an osculating line of the hypersurface $X_t$ at the point $q$. Hence, the image of a general point $p \in C_t$ is contained in the $1$-osculating locus of $\Lambda_{X_t}$ the hypersurface
\[
\Lambda_{X_t} = \bigl\{x \in X_t:
 \text{ there exists a line } \ell \text{ s.t. } \ell \cap X_t = 6 \cdot x \bigr\}.
\]
Since the $1$-osculating locus is closed, $h_t(C_t)$ is contained in $\Lambda_{X_t}$. The genus bound of the curve $C_t$ will then follow using Pacienza's technique \cite[Section 4.1]{Pac03}.
\end{comment}

\begin{lemma}[\textup{cf. \cite[Lemma~2.11]{CR23}}]
\label{lem:fiber-of-restriction} Let
$Z\subseteq\bP^4$ be a subvariety containing $q$. Suppose that, under the
natural map
\[
T_{\wt{\sC}/\bP^4}|_{(t,p)}
\ \longrightarrow\ 
T_{\sX/\bP^4}|_{(t,q)}
\ \simeq\  H^0\bigl(\bP^4,\cI_q(6)\bigr),
\]
the image contains $H^0\bigl(\bP^4,\cI_Z(6)\bigr)$. Then
\[
(t,q)+H^0\bigl(\bP^4,\cI_Z(6)\bigr)
\ \subseteq\  h(\wt{\sC}).
\]
\end{lemma}

In other words, if a sextic hypersurface through $q$ parametrized by $T$ belongs to $h(\wt{\sC})$, then the same holds after adding any sextic polynomial vanishing on $Z$.

\begin{lemma}[{cf. \cite[Lemma~3.9]{Yeong25}}]
\label{lem:contain-in-osc-locus}
For a general point $p\in \tC$, one has $\ell_p\cap X = 6q$ as schemes.
\end{lemma}

\begin{proof}
Let $F$ be the sextic polynomial defining $X=\sX_t$, and set
$\ell\coloneqq\ell_p$. Consider the restriction map
\[
\beta\ \colon\ 
T_{\wt{\sC}/\bP^4}|_{(t,p)}
\ \longrightarrow\ 
H^0(\bP^4,\cI_q(6))
\ \longrightarrow\ 
H^0\bigl(\ell,\cI_q(6)\bigr).
\] defined in \Cref{eq:comm-diag}. By \Cref{lem:ideal-in-relative-tangent}, one has
\[
H^0(\bP^4,\cI_\ell(6))
\ \subseteq\ 
T_{\wt{\sC}/\bP^4}|_{(t,p)}.
\]
Since $T_{\wt{\sC}/\bP^4}|_{(t,p)}$ has codimension $2$ in
$H^0(\bP^4,\cI_q(6))$, it follows that
\[
\dim\im\beta
\ =\ 
h^0\bigl(\ell,\cO_\ell(6)(-q)\bigr)-2
\ =\ 
4.
\]

Let $G\coloneqq\GL(5)$, as in \Cref{sec:deformation}, and let
$H\coloneqq\Stab_G(\ell,q)^\circ$ be the identity component of the
stabilizer of $(\ell,q)$, with Lie algebra
$\mathfrak h\coloneqq\Lie(H)$. The induced action of $H$ on $\ell$
gives an infinitesimal action of $\mathfrak h$ on
$H^0(\ell,\cI_q(6))$. Set $s\coloneqq F|_\ell$ and
\[
K\ \coloneqq\ 
\left\langle s,\mathfrak h\cdot s\right\rangle
 \ \subseteq \ H^0\bigl(\ell,\cI_q(6)\bigr).
\]
For every $\xi\in\mathfrak h$, the corresponding infinitesimal
$G$-action gives a tangent vector in
$T_{\wt{\sC}/\bP^4}|_{(t,p)}$, since $\xi$ fixes $q$, and its image
under $\beta$ is
\[
(\xi\cdot F)|_\ell \ =\ \xi\cdot s.
\]
Similarly, the tangent vector induced by scaling $F$ maps to $s$.
Hence $K\subseteq\im\beta$.

The first order deformation space of $\ell$ in $\bP^4$ fixing
$q$ is $H^0\bigl(\ell,N_{\ell/\bP^4}(-q)\bigr)$. By
$G$-equivariance, differentiating the restriction of $F$ along such
deformations induces a map
\[
\theta_F\colon
H^0\bigl(\ell,N_{\ell/\bP^4}(-q)\bigr) \ \simeq \ H^0\big(\ell,\cO_{\ell}^{\oplus3}\big) \ \longrightarrow\ 
\im\beta/K.
\]
By \Cref{lem:fiber-of-restriction}, we may choose $F$ generally in the affine space $F+H^0(\bP^4,\cI_\ell(6))$, while keeping $F|_\ell$, and hence $K$,
fixed. This allows the first normal derivative of $F$ along $\ell$ to
vary freely, and hence $\theta_F$ is injective for general $F$. Consequently,
\[
4\ =\ \dim\im\beta\ \geq\ \dim K+3,
\]
and hence $\dim K\leq1$.

Since $X$ contains no lines, one has $s\neq0$ and hence $K=\langle s\rangle$. Therefore the line spanned by $s$ is invariant under $H$, and hence so is the divisor $\operatorname{div}(s)$ on $\ell$. However, since
$\im \big(H\to \Aut(\ell,p)\big)\simeq \mathbf G_a\rtimes\mathbf G_m$ acts transitively on
$\ell\setminus\{q\}$, every finite $G$-invariant divisor on $\ell$ is supported at $q$, and hence
\[\divi(s)\ =\ 6q. \qedhere
\]
\end{proof}

\begin{prop}\label{prop:line-scroll}
If $\cE$ is of rank $2$, then
\[
2g-2
\ \geq\ 
4\deg f^*\cO_{\bP^4}(1).
\]
\end{prop}

\begin{proof}
Let $\Lambda_X\subseteq X$ denote the locus swept out by the osculating
lines of $X$. Then by \Cref{lem:contain-in-osc-locus}, one has
$f(\tC)\subseteq\Lambda_X$. Since both are curves, $f(\tC)$ is an irreducible component of $\Lambda_X$. 

Consider the incidence curve
\[
\widetilde{\Delta}_X
\coloneqq
\big\{
(x,\ell)\in X\times\bG(1,4) \ \bigm| \ 
x\in\ell
\text{ and }
\ell\cap X=6x
\big\}.
\]
By \cite[Section~4.1]{Pac03}, the natural morphism
$\widetilde{\Delta}_X\to\Lambda_X$ is a desingularization, and
\[
K_{\widetilde{\Delta}_X}
\ =\ 
(4H+11L)|_{\widetilde{\Delta}_X},
\]
where $H$ and $L$ denote the pullbacks of
$\cO_{\bP^4}(1)$ and $\cO_{\bG(1,4)}(1)$, respectively. It is worth noting that although \cite{Pac03} assumes that $n\geq 6$, the construction of \cite[Section~4.1]{Pac03} applies whenever $d\leq2(n-1)$, while the natural projection to the osculating locus is generically injective whenever $d>n$. Since $n=4$ and $d=6$, both conditions hold.

The map $p\mapsto(f(p),\ell_p)$ identifies $\tC$ with the connected component of $\widetilde{\Delta}_X$ lying over $f(\tC)$. Therefore,
\[
2g-2
\ =\ 
\deg K_{\tC}
\ =\ 
4\deg f^*\cO_{\bP^4}(1)
+
11\deg\phi^*\cO_{\bG(1,4)}(1)
\ \geq\ 
4\deg f^*\cO_{\bP^4}(1).\qedhere\]
\end{proof}

\subsection{Rank two quotient and moving planes}

Let $\cN$ be a rank two quotient bundle of $\cM$. Then \Cref{eq:quotient-bundle} gives rise to a morphism \[\phi:\tC \ \longrightarrow \ \bG\coloneqq \Gr(V,3), \]
such that $f(p)\in \Lambda_p$ for any $p \in \tC$, where $\Lambda_p$ is the plane in $\bP^4$ represented by $\phi(p)$. Let
\[0\ \longrightarrow \ \cS \ \longrightarrow \ V\otimes \cO_{\bG} \ \longrightarrow \ \cQ \ \longrightarrow \ 0\] be the tautological sequence on $\bG$, and $\cO_{\bG}(1)=\det \cQ$ be the Pl\"{u}cker line bundle.

\begin{lemma}\label{lem:pullback-degree-Plucker}
One has $\deg\phi^*\cO_{\bG}(1)\leq 2g-2$.
\end{lemma}

\begin{proof}
By construction, $\phi^*\cO_{\bG}(1)\cong\det\cE$, and hence
\[
\deg\phi^*\cO_{\bG}(1)\ =\ \deg\cL+\deg\cN.
\]
Since $\cN$ is a subsheaf of $N_{f/X}$ of the same rank, one has
\[
\deg\cN\ \leq\ \deg N_{f/X}
\ =\ \deg f^*T_X-\deg T_{\tC}
\ =\ -\deg f^*\cO_X(1)-(2-2g).
\]
Since $\cL=f^*\cO_X(1)$, the desired inequality follows.
\end{proof}

\begin{prop}\label{prop:constant-plane}
If $\phi$ is constant, then $2g-2\geq \deg f^*\cO_{\bP^4}(1)$.
\end{prop}

\begin{proof}
Let $\Lambda\subseteq\bP^4$ be the plane corresponding to the image of $\phi$. Then $C\subseteq X\cap\Lambda$. By \Cref{lem:no-low-degree-plane-curve}, the plane section $X\cap\Lambda$ is integral, and hence $C=X\cap\Lambda$. It follows from \Cref{prop:normalization-of-plane-curve} that $g=g(\wt{C})\geq 4$. Therefore, $2g-2\geq 6=\deg f^*\cO_{\bP^4}(1)$.
\end{proof}

We may now assume that $\phi$ is finite onto its image. We conclude this section by treating a special case; the general case will be addressed in the next section, where the moduli of stable maps and the geometry of fibered surfaces enter the argument.

\begin{prop}\label{prop:singular-plane-section}
If $f(p)$ is a singular point of $\Lambda_p\cap X$ for every $p\in \tC$, then
\[
\deg f^*\cO_{\bP^4}(1)\ \leq\  c(2g-2)
\]
for some constant $c$ depending only on $X$.
\end{prop}

\begin{proof}
    Consider the incidence correspondence \[
 \Sigma^{\sing}_X\ \coloneqq
 \bigl\{(x,\Lambda)\in X\times \bG:
 x\in\Lambda,\ \Lambda\subset T_xX\bigr\},\]
and let $q_1:\Sigma^{\sing}_X\to X$ and $q_2:\Sigma^{\sing}_X\to \bG$ be the two projections.

\begin{lemma}
The morphism $q_2:\Sigma^{\sing}_X\to \bG$ is finite.
\end{lemma}

\begin{proof}
For any plane $\Lambda\subseteq \bP^4$, the fiber $q_2^{-1}([\Lambda])$ is the singular locus of the integral plane curve $X\cap\Lambda$. Thus $q_2$ is quasi-finite, and hence finite by the properness.
\end{proof}

As a consequence, $q_2^*\cO_{\bG}(1)$ is ample on $\Sigma^{\sing}_X$. Therefore one can choose a positive integer $c\gg0$ such that 
\[
 q_2^*\cO_\bG(c)\otimes q_1^*\cO_X(-1)
\]
is nef. Since $f(p)$ is a singularity of $\Lambda_p\cap X$, the morphism $(f,\phi):\tC\to X\times \bG$ factors through $\Sigma^{\sing}_X$. Pulling back the nef
line bundle to $\tC$, by \Cref{lem:pullback-degree-Plucker}, gives
\[
 c\cdot (2g-2)-\deg f^*\cO_{\bP^4}(1)\geq0. \qedhere
\]
\end{proof}

\bigskip

\section{Generic plane scroll}\label{sec:moving-plane-case}

In this section, we establish a uniform positive lower bound for the ratio
$\deg\omega_{\tC}/\deg f^*\cO_{\bP^4}(1)$
under the assumptions that $\phi\colon\tC\to\bG$ is nonconstant and that
$f(p)$ is a smooth point of $\Lambda_p\cap X$ for general $p\in\tC$.

\medskip

Let $\phi\colon\tC\to\bG$ be the morphism induced by the quotient
$V\otimes\cO_{\tC}\twoheadrightarrow\cE$ in
\Cref{eq:quotient-bundle}, and retain the notation introduced there.
Then $\bP_{\bG}(\cQ)\to\bG$ is the universal family of planes in $\bP^4$, while $\cC_{\bG}\to\bG$ is the universal family of plane
sections of $X$. These families fit into the following commutative
diagram:
\begin{equation}
\begin{tikzcd}[ampersand replacement=\&]
	{\cC_{\tC}} \&\& {\bP_{\tC}(\cE)} \&\& \tC \\
	{\cC_{\bG}} \&\& {\bP_{\bG}(\cQ)} \&\& \bG \\
	{\bG\times X} \&\& {\bG\times \bP^4} \&\& \bG
	\arrow[from=1-1, to=1-3]
	\arrow[from=1-1, to=2-1]
	\arrow[from=1-3, to=1-5]
	\arrow[from=1-3, to=2-3]
	\arrow["\phi"', from=1-5, to=2-5]
	\arrow[hook, from=2-1, to=2-3]
	\arrow[hook, from=2-1, to=3-1]
	\arrow[from=2-3, to=2-5]
	\arrow[hook, from=2-3, to=3-3]
	\arrow["{=}"{marking, allow upside down}, draw=none,
		from=2-5, to=3-5]
	\arrow[hook, from=3-1, to=3-3]
	\arrow[from=3-3, to=3-5]
	\arrow["\lrcorner"{description, very near start}, draw=none,
		from=2-3, to=1-1]
	\arrow["\lrcorner"{description, very near start}, draw=none,
		from=2-5, to=1-3]
\end{tikzcd}.
\end{equation}

Choose a finite stratification $\bG=\coprod_i\bG_i$ by smooth irreducible
locally closed subsets such that, for each $i$, the restricted universal
family
$\sC_i\coloneqq\sC_{\bG}\times_{\bG}\bG_i\to\bG_i$
admits a simultaneous normalization
\[
\wt{\sC}_i\longrightarrow\sC_i\longrightarrow\bG_i.
\]
By \Cref{lem:no-low-degree-plane-curve}, the fibers of
$\sC_{\bG}\to\bG$ are integral. Consequently,
$\wt{\sC}_i\to\bG_i$ is a smooth proper family of connected curves.

Let $\bG_i$ be the stratum containing the image of the generic point
$\eta_{\tC}$ under $\phi$, and let $g_i$ denote the genus of the fibers of
$\wt{\sC}_i\to\bG_i$. By
\Cref{prop:normalization-of-plane-curve}, one has $g_i\geq4$. The family of maps

\[\begin{tikzcd}[ampersand replacement=\&]
	{\wt{\sC}_i} \& {\sC_i} \& X \\
	{\bG_i}
	\arrow[from=1-1, to=1-2]
	\arrow[from=1-1, to=2-1]
	\arrow[from=1-2, to=1-3]
\end{tikzcd}\]
induces a morphism
$\bG_i\to\ove{\MM}_{g_i}(X,\beta)$
to the moduli stack of Kontsevich stable maps of genus $g_i$ to $X$ and
class $\beta\in H_2(X,\bZ)$.

Let $\ove{\bG}_i$ be the closure of $\bG_i$ in $\bG$. Since the generic
point of $\tC$ is mapped to $\bG_i$, the morphism
$\phi\colon\tC\to\bG$ factors through $\ove{\bG}_i$. Let $\cG_i$ be the
normalization of the closure of the graph of the rational map
\[
\ove{\bG}_i\ \dashrightarrow\ \ove{\MM}_{g_i}(X,\beta).
\]
Then $\cG_i$ is a proper Deligne--Mumford stack equipped with a proper birational morphism to $\ove{\bG}_i$. Choose a finite surjective morphism $\bH_i\to\cG_i$ from a smooth projective variety $\bH_i$. This choice depends only on the fixed stratum $\bG_i$, and not on $C$. The composition
\[
\pi_i\colon\bH_i\ \longrightarrow\ \ove{\bG}_i
\]
is a generically finite projective morphism. Up to further stratification of $\ove{\bG}_i$ and repeating the above construction, one may assume that $\pi$ is finite \'{e}tale over $\bG_i$. Let $d_i$ be the degree of $\pi$. Let $B_i\subseteq\bH_i$ be the union of the non-\'{e}tale locus of $\pi_i$
and the locus over which the \emph{stabilized universal curve}
(cf. \Cref{rem:stabilized-universal-curve}) is singular, endowed with
its reduced scheme structure. Choose an effective Cartier divisor $E_i$ on $\bH_i$ whose support is $\Exc(\pi_i)$ and a very ample line bundle \[\cL_i\ \coloneqq \ \pi_i^*\cO_{\ove{\bG}_i}(m_i)\otimes \cO_{\bH_i}(-E_i)\] such that $\cI_{B_i}\otimes\cL_i$ is globally generated.

\begin{rem}\label{rem:stabilized-universal-curve}
Here, the \emph{stabilized universal curve} refers to stabilization as a family of curves, rather than as a family of maps. More precisely, for the pullback of universal domain curves
$\sC_{\ove{\bH}_i}\to\ove{\bH}_i$, we forget the maps and take its Deligne--Mumford stabilization.
\end{rem}

Let $C'$ be the normalization of a reduced irreducible component of
$\tC\times_{\ove{\bG}_i}\bH_i$ that dominates $\tC$. Denote the induced
morphisms by
\[
\tau\colon C'\longrightarrow\tC,
\qquad
u\colon C'\longrightarrow\bH_i.
\]
Let $r\coloneqq\deg\tau\leq d_i$, and let $R_{\tau}\subseteq C'$ be the ramification divisor of $\tau$. Let $g'$ be the genus of $C'$. Since $\phi$ is nonconstant, $\deg\phi^*\cO_{\bG}(1)>0$, and thus \Cref{lem:pullback-degree-Plucker} implies \[g'\ \geq\  g\ \geq\ 2.\] 
\begin{lemma}\label{lem:fixed-boundary-control}
One has
\[
\deg u^*\cL_i
\ \leq\ 
m_i r\deg\phi^*\cO_{\bG}(1) \ \leq \ m_ir(2g-2).
\]
\end{lemma}

\begin{proof}
Let $\eta_{C'}$ be the generic point of $C'$. Since
$\pi_i\circ u(\eta_{C'})\in\bG_i$ and $\pi_i\colon\bH_i\to\ove{\bG}_i$
is finite étale over $\bG_i$, one has $u(C')\not\subseteq E_i$.
Therefore, by construction,
\[
\deg u^*\cL_i
 \ \leq\ 
m_i\deg(\pi_i\circ u)^*\cO_{\bG}(1)
\ =\ 
m_i\deg(\phi\circ\tau)^*\cO_{\bG}(1)
\ =\ 
m_i r\deg\phi^*\cO_{\bG}(1).
\] The second inequality in the statement follows from \Cref{lem:pullback-degree-Plucker}
\end{proof}

\smallskip

Let $\pi^{\univ}:\sC^{\univ}\to \ove{\MM}_{g_i}(X,\beta)$ be the universal family of prestable curves. Then the pullback to $\bH_i$ and $C'$ fits into the following diagram:
\[\begin{tikzcd}[ampersand replacement=\&]
	S' \&\& {\cC_{i}} \&\& {\sC^{\univ}} \& X \\
	{C'} \&\& {\bH_i} \&\& {\ove{\MM}_{g_i}(X,\beta)}
	\arrow["{\psi_i}", from=1-1, to=1-3]
	\arrow[from=1-1, to=2-1]
	\arrow[from=1-3, to=1-5]
	\arrow["{\nu_i}"{description}, bend left = 18pt, from=1-3, to=1-6]
	\arrow["{q_i}"', from=1-3, to=2-3]
	\arrow[from=1-5, to=1-6]
	\arrow[from=1-5, to=2-5]
	\arrow[from=2-1, to=2-3]
	\arrow[from=2-3, to=2-5]
\end{tikzcd}.\]

\begin{lemma}\label{lem:DM-stable}
Every fiber of $q_i$ is Deligne--Mumford stable.
\end{lemma}

\begin{proof}
Let $h\in \bH_i$ be a point, $C_h$ be the $q_i$-fiber over $h$ with a morphism $\nu_h\colon C_h\to X$, and
let $D_h=X\cap\Lambda_h$. By construction, $\nu_h(C_h)\subseteq D_h$ and
$(\nu_h)_*[C_h]=[D_h]$. Since $D_h$ is integral, there is a unique
component of $C_h$ mapping nonconstantly to $X$, and it maps birationally onto $D_h$. By
\Cref{prop:normalization-of-plane-curve}, its normalization has genus at least $4$. Therefore, every rational component of $C_h$ is contracted by $\nu_h$, and hence by the stability condition, each such component has at least three special points. Thus $C_h$ is Deligne--Mumford stable.
\end{proof}

Let
$\rho\colon S\to S'$ be its minimal resolution. As fibers of $S'\to C'$ have at worst nodal singularities, $S'$ has at worst $A_n$-singularities, and hence $\rho$ is crepant and $S\to C'$ is a relatively minimal semistable fibration of genus $g_i$.

By assumption, $f(p)$ is a smooth point of
$X\cap\Lambda_p=(\sC_{\bG})_{\phi(p)}$ for general $p\in\tC$. It therefore has a unique lift to the normalization of the generic plane section. After the base change to $C'$, this gives a rational section of $\pi':S'\to C'$, which extends to a section
\[\sigma'\colon C'\ \longrightarrow \  S'\] by the properness. Composing with the stabilization and lifting through the proper resolution $\rho$ gives a section $\Sigma\subseteq S$.

\begin{cor}\label{cor:bounds-on-curves}
Let $s$ be the number of singular fibers of $\pi'$. Then the following
inequalities hold:
\begin{enumerate}\setlength{\itemsep}{0.3em}
    \item[\textup{(1)}]
    $(K_{S/C'}\cdot\Sigma)
    \ \leq\ 
    (2g_i-1)(2g'-2+s)$;
    \item[\textup{(2)}]
    $s
    \ \leq\ 
    m_i r\deg\phi^*\cO_{\bG}(1)$;
    \item[\textup{(3)}]
    $\deg R_\tau
    \ \leq\ 
    (d_i-1)m_i r\deg\phi^*\cO_{\bG}(1)$.
\end{enumerate}
\end{cor}

\begin{proof}
The first inequality follows immediately from \Cref{lem:log-BMY}, so it
suffices to prove \textup{(2)} and \textup{(3)}. By the global generation
of $\cI_{B_i}\otimes\cL_i$, we may choose a section
$s_i\in H^0(\bH_i,\cI_{B_i}\otimes\cL_i)$ that does not vanish at
$u(\eta_{C'})$. Let $D_i$ be its zero divisor. Then
$B_i\subseteq\operatorname{Supp}(D_i)$.

Since every singular fiber of $\pi'$ lies over a point of $u^{-1}(B_i)$, one has
\[
s \ \leq\ 
\deg u^*D_i
\ =\ 
\deg u^*\cL_i
\ \leq\ 
m_i r\deg\phi^*\cO_{\bG}(1),
\]
where the last inequality follows from
\Cref{lem:fixed-boundary-control}. This proves (2). Moreover, since $\tau$ is \'{e}tale away from $u^{-1}(B_i)$, one has 
$\Supp R_\tau \subseteq\Supp(u^*D_i)$, and consequently 
\[
\deg R_\tau
\ \leq\ 
(r-1)\deg u^*D_i
\ \leq\ 
(d_i-1)m_i r\deg\phi^*\cO_{\bG}(1),
\]
which proves \textup{(3)}.
\end{proof}

\begin{lemma}\label{lem:ambient-canonical-comparison}
There exists a constant $b_i>0$, depending only on
$\bH_i$ and $\cL_i$, such that
\[
r\deg f^*\cO_{\bP^4}(1)
\ \leq\ 
3(K_{S/C'}\cdot\Sigma)+b_i\deg u^*\cL_i.
\]
\end{lemma}

\begin{proof}
Consider the line bundle
\[
\mathscr M_i
\ \coloneqq\ 
\mathscr  \omega_{\cC_i /\bH_i}^{\otimes 3}\otimes\mathscr \nu_i^*\cO_X(-1)
\]
on $\cC_i$. For any fiber $F$ of $\cC_i\to \bH_i$, by \Cref{lem:DM-stable} and $g_i\geq 4$, $\mathscr M_i|_F$ is very ample and $H^1(F,\mathscr M_i|_F)=0$. Hence $q_{i*}\mathscr M_i$ is a locally free sheaf and the evaluation map
\[
q_i^*q_{i*}\mathscr M_i\longrightarrow\mathscr M_i
\]
is surjective. Since $\cL_i$ is ample, there exists an integer $b_i>0$, depending only on $\cH_i$ and $\cL_i$, such that
\[
q_{i*}\mathscr M_i\otimes\cL_i^{\otimes b_i}\ = \ q_{i*}(\mathscr M_i\otimes q_i^*\cL_i^{\otimes b_i})
\]
is globally generated. It follows that
\[
0
\ \leq\ 
\deg(\psi_i\circ \sigma')^*
 \bigl(\mathscr M_i\otimes q_i^*\cL_i^{\otimes b_i}\bigr)
=
3\deg(\psi_i\circ \sigma')^* \omega_{\cC_i /\bH_i}
-\deg(\nu_i\circ \psi_i\circ \sigma')^*\cO_X(1)
+b_i\deg u^*\cL_i.
\] Observe that the morphisms $\nu_i\circ\sigma'$ and $f\circ\tau$ agree. Thus one has 
\[
\deg(\nu_i\circ \psi_i\circ \sigma')^*\cO_X(1)
\ =\ 
\deg(f\circ\tau)^*\cO_X(1)
=
r\deg f^*\cO_X(1).
\]
Moreover, the resolution $\rho:S\to S'$ is crepant, so
\[
\deg(\sigma')^* \omega_{\cC_i /\bH_i}
\ =\ 
(K_{S/C'}\cdot\Sigma).
\]
Combining these equalities gives the desired result.
\end{proof}

\begin{prop}\label{prop:generic-moving-plane}
There exists a constant $c_X>0$,
depending only on $X$, such that
\[
\deg f^*\cO_{\bP^4}(1)\ \leq\  c_X(2g-2).
\]
\end{prop}

\begin{proof}
By Riemann--Hurwitz and \Cref{cor:bounds-on-curves}(3),
\[
\begin{split}
   2g'-2
\ & =\ r(2g-2)+\deg R_\tau
\ \leq\  r(2g-2)+(d_i-1)m_ir\deg\phi^*\cO_{\bG}(1)\\
& \leq\ r\big(m_i(d_i-1)+1\big)(2g-2) . 
\end{split}
\] Combining \Cref{lem:fixed-boundary-control}, \Cref{cor:bounds-on-curves}(1)(2), \Cref{lem:ambient-canonical-comparison} and \Cref{lem:pullback-degree-Plucker} therefore
gives
\begin{equation}\nonumber
\begin{split}
    r\deg f^*\cO_{\bP^4}(1)\ 
& \leq\ 3(K_{S/C'}\cdot\Sigma)+b_i\deg u^*\cL_i \\ 
& \leq\ 3(2g_i-1)(2g'-2+s) + b_im_ir(2g-2)\\
& \leq\ 57(2g'-2)+57s+b_im_ir(2g-2) \\
& \leq\ r\big(57m_i(d_i-1)+57+57m_i+b_im_i\big)(2g-2) \\
& =\ r(57m_id_i+57+b_im_i)(2g-2), 
\end{split}
\end{equation} where we use the fact that $g_i\leq10$. As there are only finitely many strata $\bG_i$, taking the maximum of the corresponding constants gives a constant $c_X$ depending only on $X$.
\end{proof}

\smallskip

\begin{proof}[Proof of \Cref{thm:main-intro}]
Let $f\colon\widetilde C\to X$ be the normalization of an integral
curve $C$ on $X$. If $\rk\cE=2$, the result follows from
\Cref{prop:line-scroll}. Suppose that $\rk\cE=3$. According to whether
$\phi$ is constant, the points $f(p)$ are always singular on
$X\cap\Lambda_p$, or neither holds, the result follows respectively
from \Cref{prop:constant-plane}, \Cref{prop:singular-plane-section}, and
\Cref{prop:generic-moving-plane}. Taking the maximum of the resulting
constants proves that $X$ is algebraically hyperbolic.
\end{proof}

\smallskip

\bibliographystyle{alpha}
\bibliography{citation}

\end{document}